\documentclass[11pt]{amsart}

\usepackage{latexsym}
\usepackage{amssymb,amsfonts,amsmath,amsthm}
\usepackage{graphicx}
\usepackage{enumerate}
\usepackage{mdwlist}
\usepackage[margin=1.3in]{geometry}

\newtheorem{theorem}{Theorem}  %[chapter]

\newtheorem{lemma}[theorem]{Lemma}
\newtheorem{observation}[theorem]{Observation}
\newtheorem{claim}[theorem]{Claim}
\newtheorem{corollary}[theorem]{Corollary}
\newtheorem{construction}[theorem]{Construction}
\newtheorem{remark}[theorem]{Remark}
\newtheorem{example}[theorem]{Example}

\newtheorem{proposition}[theorem]{Proposition}
\newtheorem{conjecture}[theorem]{Conjecture}
\newtheorem{definition}[theorem]{Definition}

\newtheorem{question}[theorem]{Question}

 {\begin{list}{}%
         {\setlength{\leftmargin}{#1}}%
         \item[]%
 }
 {\end{list}}

\allowdisplaybreaks 
\def\COMMENT#1{}
\def\TASK#1{}

\numberwithin{theorem}{section}
\numberwithin{equation}{section}

\newdimen\margin   % needed for macros \textdisplay & \ltextdisplay
\def\textno#1&#2\par{%
   \margin=\hsize
   \advance\margin by -4\parindent
          \setbox1=\hbox{\sl#1}%
   \ifdim\wd1 < \margin
      $$\box1\eqno#2$$%
   \else
      \bigbreak
      \hbox to \hsize{\indent$\vcenter{\advance\hsize by -3\parindent
      \it\noindent#1}\hfil#2$}%
      \bigbreak
   \fi}

\title[Regular subgraphs of uniform hypergraphs]{Regular subgraphs of uniform hypergraphs}

\author{Jaehoon Kim}
\thanks{ The research leading to these results was partially supported by the European Research Council under the European Union's Seventh Framework Programme (FP/2007--2013) / ERC Grant Agreements no. 306349.}

\begin{document}

\begin{abstract}
We prove that for every integer $r\geq 2$, an $n$-vertex $k$-uniform hypergraph $H$ containing no $r$-regular subgraphs has at most $(1+o(1)){{n-1}\choose{k-1}}$ edges if $k\geq r+1$ and $n$ is sufficiently large. Moreover, if $r\in\{3,4\}$, $r\mid k$ and $k,n$ are both sufficiently large, then the maximum number of edges in an $n$-vertex $k$-uniform hypergraph containing no $r$-regular subgraphs is exactly ${{n-1} \choose {k-1}}$, with equality only if all edges contain a specific vertex $v$. We also ask some related questions.\end{abstract}

\date{\today}

\maketitle 

\section{Introduction}\label{Introduction}

What are the graphs containing no $r$-regular subgraphs? For $r=2$, the answer is easy, they are forests. However, the question becomes much harder when $r$ is larger than two. Complete characterizations of graphs with no $r$-regular subgraphs seem impossible even for the case $r=3$. So it is natural to ask how many edges can a graph with no $r$-regular subgraphs have. Pyber~\cite{Pyber1985} showed that there exists a constant $c_r$ such that all $n$-vertex graphs with at least $c_r n\log{n}$ edges have an $r$-regular subgraph. On the other hand, Pyber,  R\"{o}dl and  Szemer\'{e}di~\cite{PRS1995} proved that there exists a graph with $\Omega(n \log \log{n})$ edges having no $r$-regular subgraphs for any $r\geq 3$. The gap between the two bounds still remains open. 

It is also natural to consider the same question for hypergraphs, both uniform and non-uniform hypergraphs. Mubayi and  Verstra\"{e}te~\cite{MV2009} proved that for every even integer $k\geq 4$, there exists $n_k$ such that for $n\geq n_k$, each $n$-vertex $k$-uniform hypergraph $H$ with no $2$-regular subgraphs has at most ${n-1\choose k-1}$ edges, and equality holds if and only if $H$ is a {\em full $k$-star}, that is, a $k$-uniform hypergraph consisting of all possible edges of size $k$ containing a given vertex. For non-uniform hypergraphs, it is easy to see that an $n$-vertex hypergraph $H$ with no $r$-regular subgraphs has at most $2^{n-1}+r-2$ edges. One example for the equality is a {\em full star} (that is a hypergraph consisting of all possible edges containing a given vertex) with additional $r-2$ smallest edges not containing the given vertex. The author and Kostochka~\cite{KK} proved that if $n\geq 425$ and $n>r$, a hypergraph $H$ with no $r$-regular subgraphs contains $2^{n-1}+r-2$ edges only if $H$ is a {\em full star} with $r-2$ additional edges. 
One can ask a similar question for linear hypergraphs. Dellamonica et al. \cite{DHLNPRSV} showed that the maximum number of edges in a linear $3$-uniform hypergraph with no two-regular subgraphs is $\Omega(n\log{n})$ and $O(n^{3/2}(\log{n})^5)$ and they asked whether every linear $3$-uniform hypergraph with no $3$-regular subgraphs has at most $o(n^2)$ edges. In Section \ref{example}, we confirm that this is true.

In this paper, we consider $k$-uniform hypergraphs with no $r$-regular subgraphs. The following two theorems are main results of this paper.

\begin{theorem}\label{asymptotic theorem}
Let $k,r$ be two integers with $r\geq 2, k\geq r+1$. Then there exists $n_k$ such that for $n>n_k$ any $n$-vertex $k$-uniform hypergraph $H$ with no $r$-regular subgraphs has at most $(1+o(1)){{n-1}\choose {k-1}}$ edges. Moreover, if $k\geq 2r+1$ and $|H|\geq (1-n^{-\frac{2}{3r^2}}){{n-1}\choose {k-1}}$, then there exists a vertex $v$ which belongs to at least $(1- n^{-\frac{1}{6r^2}}){{n-1}\choose {k-1}}$ edges.
\end{theorem}

\begin{theorem}\label{main theorem}
Let $k,r$ be two integers with $r\in \{3,4\}$, $k\geq 140 r$ and $r \mid k$. Then there exists $n_k$ such that for $n>n_k$ any $n$-vertex $k$-uniform hypergraph $H$ with no $r$-regular subgraphs has at most ${{n-1}\choose {k-1}}$ edges. Moreover, the equality holds if and only if $H$ is a full $k$-star.
\end{theorem}

Our proofs of theorems develop ideas in \cite{MV2009}. In Section \ref{section asymptotic} and Section \ref{section stability}, we prove Theorem \ref{asymptotic} and Theorem \ref{stability} which together imply Theorem \ref{asymptotic theorem}. 
 In Section \ref{exact result}, we prove Theorem \ref{main theorem}. In  Section \ref{example}, we show some examples which somewhat explain the necessity of each condition in each theorem and we also pose some further questions.

\section{Preliminaries}\label{Preliminaries}

For a positive integer $N$ we write $[N]$ to denote the set $\{1,\dots,N\}$. We say $H$ has an {\em $r$-regular subgraph} if there exists a collection of edges in $E(H)$ which all together cover each vertex in a nonempty set exactly $r$-times and no other vertices. We write $V(H)$ and $E(H)$ for the set of vertices and the set of edges in a hypergraph $H$, respectively. We denote the size of $H$ by $|H|:=|E(H)|$. For a hypergraph $H$ and vertex $v$, $d_{H}(x):=|\{e\in E(H): x\in e\}|.$
$\log$ denotes $\log_2$ and {\em $s$-set} denotes a set of size $s$.
For a hypergraph $H$ and a vertex set $D$, we define $H_D$, the {\em link graph of $D$ in $H$} by $V(H_D):=V(H),~ E(H_D):=\{e\setminus D: e \in E(H), D\subseteq e\}.$ If $D=\{x\}$, we denote the link graph by $H_{x}$ instead of $H_{D}$.

First, we introduce the following simple observation which we use several times in the paper.
\begin{observation} \label{matching}
For $t>1$ and $n\geq 2k$, if an $n$-vertex $k$-uniform hypergraph $H$ has at least $t{{n-1}\choose {k-1}}$ edges, then $H$ contains a matching of size $\max\{2,\lceil \frac{t}{k} \rceil\}$.
\end{observation}
\begin{proof}
If $t\leq 2k$, it is obvious by Erd\H{o}s-Ko-Rado theorem \cite{EKR}.
Assume $t > 2k$. We greedily choose disjoint edges from $H$. If we choose $\ell<\lceil \frac{t}{k} \rceil$ disjoint edges, the number of edges intersecting at least one of them is at most $\ell k {{n-1}\choose {k-1}} < t{{n-1}\choose {k-1}}$. Thus we can choose an edge disjoint from all previous ones to extend the matching. We can do this until we get $\lceil \frac{t}{k} \rceil$ disjoint edges to get a matching of size $\max\{2,\lceil \frac{t}{k} \rceil \}$. \end{proof}

The following is another simple observation which we will use later.

\begin{observation}\label{obs2}
Let $r,k',k$ be integers with $k=rk'$ and $\mathcal{B}\subseteq {{[k]}\choose{k'}}$ satisfying that for any $r$-equipartition $A_1,\dots, A_{r}$ of $[k]$, at least one $A_i$ belongs to $\mathcal{B}$. Then 
$$|\mathcal{B}| \geq \frac{1}{r}{{k}\choose{k'}}.$$
\end{observation}
\begin{proof}
We pick an $r$-equipartition $\mathcal{A}=(A_1,\dots,A_r)$ of $[k]$ uniformly at random. For a set $B\in {{[k]}\choose{k'}}$, we say $B\in \mathcal{A}$ if $B=A_i$ for some $i\in [r]$.
Since $\mathcal{A}$ is chosen uniformly at random, for any $k'$-set $B \in {{[k]}\choose{k'}}$ we have
$\mathbb{P}[ B \in \mathcal{A} ] = r {{k}\choose{k'}}^{-1}.$
For any $\mathcal{A}$, there exists $B\in \mathcal{B}$ which satisfy $B\in \mathcal{A}$. Thus we have 
$$\mathbb{E}[ |\{B : B\in \mathcal{A}, B\in \mathcal{B}\}| ] \geq 1.$$
On the other hand,
$$1\leq \mathbb{E}[ |\{B : B\in \mathcal{A}, B\in \mathcal{B}\}| ] = \sum_{B \in \mathbb{B}} \mathbb{P}[ B\in \mathcal{A} ] \leq r{{k}\choose{k'}}^{-1}|\mathcal{B}|.$$
Therefore $|B| \geq \frac{1}{r}{{k}\choose{k'}}.$
\end{proof}

The following is a theorem from \cite{FRR} concerning about the size of hypergraph without a matching of certain size.

\begin{theorem}\cite{FRR}\label{obs 3matching}
For $s\geq 1$ and $n\geq 4s$, if $H$ is an $n$-vertex $3$-uniform hypergraph with no matching of size $s$, then 
$$ |H| \leq {{n}\choose{3}} - {{n-s+1}\choose{3}}.$$
\end{theorem}

Now we introduce the notion of {\em sunflower}. Erd\H{o}s and Rado~\cite{ER0} introduced the following notion of sunflower in connection with some problems in Number Theory. It is also called a {\em $\Delta$-system}.

\begin{definition} \label{delta-system}
A family of $p$ sets is a $p$-sunflower if the intersections of any
two sets in the family are all the same. Let $q(k,p)$ be the least integer $q$
such that every $k$-uniform family of $q$ sets contains a $p$-sunflower.
\end{definition}
They also showed that $q(k,p)$ exists for any positive integer $k,p$. It means that if a $k$-uniform hypergraph has no $p$-sunflower, then the number of edges in the hypergraph is bounded by $q(k,p)$. In particular, they proved the following.
\begin{theorem}\cite{ER0} \label{ERDelta}
$$(p-1)^{k} \leq q(k,p) \leq (p-1)^k k! $$
\end{theorem}

They also conjectured that $q(k,p)\leq c_p^k$ for some constant $c_p$.
Abbott, Hanson, and Sauer \cite{AHS} and later F\"{u}redi and Kahn (see \cite{Erdos prob}) improved the upper bound of Theorem \ref{ERDelta}.
The following result on the topic is due to Kostochka. 

\begin{theorem}[Kostochka~\cite{Kos96}] \label{delta-system}
For $p\geq 3$ and $\alpha>1$, there exists $D(p,\alpha)$ such that
$q(k,p) \leq D(p,\alpha) k! (\frac{(\log\log\log{k})^2}{\alpha \log\log {k}})^k$.
\end{theorem}

Essentially, Theorem \ref{delta-system} implies that there exists a constant $c(p)$ such that $q(k,p) \leq \frac{k^k}{(\log\log{k})^{k/2}}$ for $k$ at least $c(p)$. By using Theorem \ref{delta-system}, we prove the following lemma which is a variation of Lemma 1 in \cite{MV2009}. Note that the proof is identical to the proof of Lemma 1 in \cite{MV2009} except the part using Theorem \ref{delta-system}.

\begin{lemma} \label{lemma 1}
There exists a constant $c(r)$ such that the following holds.
Let $k,r$ be integers and $H$ be a $k$-uniform hypergraph on $n$ vertices containing no $r$-regular subgraphs with maximum degree $\Delta=\Delta(H)$.  If $|H| \geq c(r) \Delta k$, then $$|H|\leq \frac{6n^{k/(k-1)}\Delta^{(k-2)/(k-1)}}{(\log\log{\frac{|H|}{k\Delta}})^{\frac{1}{2(k-1)}}}.$$
\end{lemma}
\begin{proof}
Let $m=\lfloor \frac{|H|}{k\Delta} \rfloor$.
Suppose $|H|\geq c(r) k\Delta$ and $$|H|> \frac{6n^{k/(k-1)}\Delta^{(k-2)/(k-1)}}{(\log\log{m})^{\frac{1}{2(k-1)}}}$$ for a contradiction. Then $m\geq c(r)$. These assumptions imply that 

$$ m^{k-1}=(\frac{|H|}{k\Delta})^{k-1} \geq \frac{6^{k-1}n^k \Delta^{k-2}}{k^{k-1}\Delta^{k-1}(\log\log{m})^{1/2}} = \frac{1}{k\Delta (\log\log{m})^{1/2}} \frac{6^{k-1}n^k}{k^{k-2}}.$$

So, we get 
\begin{equation} \label{2.1}
(k\Delta)^{m} \geq (\frac{6^{k-1}n^k}{m^{k-1} k^{k-2}(\log\log{m})^{1/2}})^m > \frac{k^{2m} m^m}{(\log\log{m})^{m/2}} (\frac{3n}{mk})^{mk} > \frac{k^{2m} m^m}{(\log\log{m})^{m/2}} { n\choose {mk}}.
\end{equation}

Now we count the matchings of size $m$ in $H$. We may greedily pick edges $e_1,e_2,\cdots, e_m$ so that edges are disjoint. At first, we have $|H|$ choices for $e_1$. In each step, we exclude all edges intersecting previously chosen edges from the list of choices. Then we exclude at most $k\Delta$ edges in each step. Thus we conclude that the number of matchings of size $m$ in $H$ is at least

\begin{equation} \label{2.2}
\frac{1}{m!} \prod_{i=0}^{m-1}(|H|-k\Delta i) = \frac{1}{m!} |H|^m \prod_{i=0}^{m-1} (1-\frac{k\Delta i}{|H|}) \geq\frac{1}{m!} |H|^m \prod_{i=0}^{m-1} (1-\frac{i}{m})\geq (k\Delta)^m.
\end{equation} 

Because the number of $mk$-sets in $V(H)$ is ${ n\choose {mk}}$, (\ref{2.1}) and (\ref{2.2}) together assert that there are at least $ \frac{k^{2m} m^m}{(\log\log{m})^{1/2}} \geq q(m,r) $ distinct matchings $M_1, M_2,\cdots, M_{q(m,r)}$ covering exactly the same set $M$ of size $mk$.
Consider the following auxiliary hypergraph $\mathcal{H}$ with $$V(\mathcal{H})= \{e \in H\},\enspace E(\mathcal{H}) = \{M_i : i=1,\cdots, q(m,r) \}.$$ Note that a vertex in $\mathcal{H}$ is an edge in $H$, and an edge in $\mathcal{H}$ is a matching of size $m$ in $H$. By Theorem \ref{delta-system} there are at least $r$ distinct matchings $M_{i_1},\cdots, M_{i_r}$ which together form an $r$-sunflower in $\mathcal{H}$. 
By the definition of $r$-sunflower, there exists a set $M'$ such that $M'=M_{i_j}\cap M_{i_{j'}}$ for any $j,j' \in [r]$ with $j\neq j'$. Then $M_{i_j} - M'$ for $j=1,2,\cdots, r$ are $r$ disjoint matchings covering the same set $ M-\bigcup_{e\in M'} e$.
 Thus $\bigcup_{j=1}^{r} (M_{i_j}-M)$ gives us an $r$-regular subgraph of $H$, it is a contradiction.\end{proof}

We also use the following theorem of Pikhurko and Verstra\"{e}te in several places.
\begin{theorem}\label{generalized C4} \cite{PV}
For $k\geq 3$, if $H$ is an $n$-vertex $k$-uniform hypergraph with at least $\frac{7}{4}{{n-1}\choose {k-1}}$ edges, then $H$ contains two pairs of sets $\{A,B\}, \{C,D\}$ so that 
$$ A\cap B = C\cap D =\emptyset,\enspace A\cup B = C\cup D.$$
\end{theorem}

Now we introduce new hypergraphs $H(k,\ell)$ and $H'(k,\ell)$ which will be useful for proving several claims later.

\begin{definition} \label{H(k,l)}
Let $A,B$ be two disjoint $(k-\ell)$-sets and $Y =\{ u_1,\dots,u_\ell, v_1,\dots, v_\ell\}$. For two nonnegative integer $k,\ell$ with $k>\ell$, we define $H(k,\ell)$ to be the $2k$-vertex $k$-uniform hypergraph on the ground set $A\cup B\cup Y$ satisfying the following,
$$E(H(k,\ell)) = \{ e\cup Z : |e\cap \{u_i,v_i\}|=1 \text{ for all $i\in[\ell]$}, |e|=\ell , Z\in \{A,B\} \}.$$
We call each of $A$ and $B$ a stationary part, and vertices in them stationary vertices. Also we call vertices in $Y$ dynamic vertices and let $V_d(H(k,\ell))$ denote $Y$.
\end{definition}

Note that if $e$ is an edge of $H(k,\ell)$, then there exist indices $i_1,\dots, i_s, j_1,\dots, j_{\ell-s}$ with $\{i_1,\dots, i_s\}\cup \{j_1,\dots, j_{\ell-s}\}=\{1,\dots,\ell\}$ such that
$$e=\{ u_{i_1},\dots, u_{i_s}, v_{j_1},\dots, v_{j_{\ell-s}}\} \cup Z\text{ for }Z\in \{A,B\}.$$ 
Then $e'=\{ v_{i_1},\dots, v_{i_s}, u_{j_1},\dots, u_{j_{\ell-s}}\} \cup Z'$ for $Z' = (A\cup B)\setminus Z$ is also an edge in $H(k,\ell)$. Thus the following holds.
\begin{equation}\label{H(k,l) disjoint}
\begin{minipage}[c]{0.9\textwidth}\em
For an edge $e$ in $H(k,\ell)$, there exists $e'\in H(k,\ell)$ with $e\cap e'=\emptyset$, $e\cup e' = V(H(k,\ell))$.
\end{minipage}
\end{equation}

\begin{definition} \label{H'(k,l)}
Let $A,B,C,D$ be four distinct $(k-\ell)$-sets satisfying $A\cup B = C\cup D$, $A\cap B= C\cap D = \emptyset$, and $Y =\{ u_1,\dots,u_\ell, v_1,\dots, v_\ell\}$. For two nonnegative integers $k,\ell$ with $k>\ell$, we define $H'(k,\ell)$ to be the $2k$-vertex $k$-uniform hypergraph on the ground set $A\cup B\cup Y$ satisfying the following,
$$E(H'(k,\ell)) = \{ e\cup Z : |e\cap \{u_i,v_i\}|=1 \text{ for all $i=1,2,\cdots,\ell$}, |e|=l , Z\in \{A,B,C,D\} \}.$$
We call each of $A,B,C$ and $D$ a stationary part, and vertices in them stationary vertices. Also we call vertices in $Y$ dynamic vertices.
\end{definition}

Note that the following holds.
\begin{equation}\label{H(k,l) regular subgraph}
\begin{minipage}[c]{0.9\textwidth}\em
Hypergraph $H(k,\ell)$ contains $2^{\ell}$ edge-disjoint matchings of size $2$ covering $V(H(k,\ell))$ and $H'(k,\ell)$ contains $2^{\ell+1}$ edge-disjoint matchings of size $2$ covering $V(H'(k,\ell))$.
\end{minipage}
\end{equation}

Indeed, $H(k,\ell)$ is a $k$-uniform hypergraph which resembles the complete $(\ell+1)$-partite $(\ell+1)$-uniform hypergraph with all parts size two. Because of the resemblance, its Turan number is related to the Turan number of $(\ell+1)$-partite $(\ell+1)$-graph. 
The lemma below is proved by Erd\H{o}s, and we use it to bound the Turan number of $H(k,\ell)$.

\begin{lemma}\cite{Erdos}\label{Erd}
Let $S$ be a set of $N$ elements $y_1,y_2,\cdots, y_N$ and let $A_i$ for $1\leq i\leq n$ be subsets of $S$.  If $\sum_{i=1}^{n} |A_i| \geq \frac{nN}{w}$ for some $w$ and $n\geq 8 w^2$, then there are $2$ distinct $A_{i_1}, A_{i_{2}}$ so that 
$$ |A_{i_1}\cap A_{i_2}| \geq \frac{N}{2w^2} .$$
\end{lemma}
\begin{corollary}\label{link graph}
Let $w,n$ be numbers satisfying $n \geq 8w^2$. 
If $H$ is a $k$-uniform hypergraph with $|H|\geq \frac{1}{kw}n\binom{n}{k-1}$, then there are two vertices $x,x'$ such that $|E(H_x)\cap E(H_{x'})|\geq \frac{1}{2w^2}\binom{n}{k-1}$ where $H_x, H_{x'}$ are link graph of $x$ and $x'$ in $H$, respectively.
\end{corollary}
\begin{proof}
Let $x_1,\dots, x_n$ be the vertices of $H$. Let $A_i=E(H_{x_i})$ be the $(k-1)$-uniform hypergraph, which is the link graph of $x_i$ in $H$. Since each $A_i$ is $(k-1)$-uniform hypergraphs, $A_i$ is a subset of $\binom{V(H)}{k-1}$. Since $\sum_{i=1}^{n}|A_i| = \sum_{i=1}^{n} d_{H}(x_i) = k|E(H)| = \frac{n\binom{n}{k-1}}{w}$, we apply Lemma~\ref{Erd} with $A_i, \binom{V(H)}{k-1}, w$ playing the role of $A_i, S, w$, respectively. Then we obtain there exist $x_{i}, x_{i'}$ with 
$|E(H_{x_i})\cap E(H_{x_{i'}})|= |A_i\cap A_{i'}|\geq \frac{1}{2w^2}\binom{n}{k-1}.$
\end{proof}
\begin{proposition} \label{ext(H(k,l))} 
Let $k,\ell \geq 0$ be integers where $k>\ell$.
Then for $n>2k$, any $k$-uniform hypergraph $H$ with $2n^{k-2^{-\ell}}$ edges contains a copy of $H(k,\ell)$ as a subgraph. Moreover, if $k\geq \ell+3$, then it also contains a copy of $H'(k,\ell)$.
\end{proposition}
\begin{proof}
We use induction on $\ell$. For $\ell=0$, assume we have an $n$-vertex $k$-uniform hypergraph $H$ with $2n^{k-1}$ edges. By Observation \ref{matching} and the fact $n^{k-1} > {{n-1}\choose {k-1}}$, we get $H(k,0)$ which is a matching of size two for any $k$. If $k\geq 3$ and $\ell=0$, then Theorem \ref{generalized C4} implies that $H$ contains $H'(k,\ell)$, which consists of two pairs of disjoint edges with the same union. For $k= 2, \ell=1$, Turan number for the cycle of length $4$ gives us the conclusion about $H(k,\ell)$. 

Assume now that every $n$-vertex $k$-uniform hypergraph with $2n^{k-2^{-\ell+1}}$ edges contains a copy of $H(k,\ell-1)$ for $n> 2(k-1)$ and $k\geq 3, \ell\geq 1$. If an $n$-vertex $k$-uniform hypergraph $H$ with $n>2k$ contains at least $2 n^{k-2^{-\ell}} \geq \frac{2(k-1)!}{n^{2^{-\ell}}} n\binom{n}{k-1}$ edges, Corollary~\ref{link graph} implies that there are two vertices $x,x'\in V(H)$ with 
$$|E(H_x)\cap E(H_{x'})| \geq \frac{1}{2}(2(k-1)!n^{-2^{-\ell}})^2\binom{n}{k-1} > 2 n^{k-1-2^{-\ell+1}}.$$
By induction hypothesis, $(k-1)$-uniform hypergraph $E(H_x)\cap E(H_{x'})$ contains $H'$, a copy of $H(k-1,\ell-1)$. Then 
$$\{ \{z\} \cup e : e\in E(H'), z\in \{x, x'\} \}$$ forms a copy of $H(k,\ell)$. Thus $H$ must contain a copy of $H(k,l)$. We get the conclusion for $H'(k,\ell)$ by the same logic.\end{proof}

\section{Approximate size of $H$} \label{section asymptotic}

In this section, we prove the following Theorem \ref{asymptotic} by showing that most of the edges in $H$ contain only one vertex of high degree. Note that we only consider the case when $r\geq 3$ because the case of $r=2$ is already done in \cite{MV2009}.
We let $\ell:=\lceil\log{r} \rceil$ and let $\alpha$ be a number which we decide later such that $0<\alpha \leq 1/2$, and we let
\begin{align}\label{def D}
D:= n^{k-1-\alpha}(\log\log{n})^{\frac{1}{4(k-1)}}.
\end{align}
We let $T$ denote the set of vertices of $H$ of degree at least $D$ and set $t:=|T|$. Since $tD \leq k |H|$, 
\begin{equation}\label{size of t} t\leq D^{-1} k |H|. \end{equation}
 We also define
$ H_i := \{e \in H : |e\cap T| = i\} \text{ for } i\leq k, $ and
$ G:=\{e\in H_1 : \nexists f \in H_1 \text{ such that } e\setminus T = f\setminus T\}.$ Then, it is obvious that $|G|\leq {{n-1}\choose{k-1}}$. Note that $r+1$ is always at least $2^{\ell-1}+2$.
\begin{theorem} \label{asymptotic}
For integer $k,r,\ell$ with $k>r\geq 3$, $\ell=\lceil \log{r}\rceil$, there exists an integer $n_k$ such that for $n\geq n_k$ any $n$-vertex $k$-uniform hypergraph $H$ with no $r$-regular subgraphs satisfies the following.
\begin{align*}
\text{If }k> 2^{\ell-1}+2,\text{ then } & |H| \leq {{n-1}\choose {k-1}} + 3n^{k-1-\frac{1}{2r^2-2}}. \\
\text{If }k= 2^{\ell-1}+2,\text{ then } & |H| \leq {{n-1}\choose {k-1}} + 3n^{k-1}(\log\log{n})^{-\frac{1}{4(k-1)}}.
\end{align*}
\end{theorem}
\begin{proof}
First we suppose the conclusion does not hold. We may assume that we have a counterexample $H$ such that $|H|$ is one more than the stated upper bound by deleting some edges if necessary and assume $n$ is large enough. 
Since $n$ is large enough, $|H|\leq n^{k-1}/k$ and (\ref{size of t}) imply 
\begin{align} \label{t size}
t \leq D^{-1}k|H|\leq n^{\alpha}(\log\log{n})^{-\frac{1}{4(k-1)} }.
\end{align} 

\begin{claim} \label{asymptotic claim}
$$|H_0| \leq \max\left\{ n^{k-1-\alpha + \frac{1}{2k^2}},  n^{k-1 + \frac{1- (k-2)\alpha}{k-1} }(\log\log{n})^{-\frac{1}{4(k-1)}}\right\}.$$
$$|H\setminus (H_0\cup H_1)| \leq n^{k-2+2\alpha}(\log\log{n})^{-\frac{1}{2(k-1)}}.$$
\end{claim}
\begin{proof}
First, we estimate $|H_0|$. Since edges in $H_0$ do not intersect $T$, the maximum degree of $H_0$ is less than $D$. We apply Lemma \ref{lemma 1} to $H_0$, then we get 
$$|H_0| \leq \max \left\{ c(r)k\Delta(H_0) ,  \frac{6n^{k/(k-1)}D^{(k-2)/(k-1)}}{(\log\log{\frac{|H|}{k D}})^{\frac{1}{2(k-1)}}} \right\}$$
where $c(r)$ is the constant from Lemma \ref{lemma 1}.
Since $n$ is large enough, 
$$c(r)k\Delta(H_0)  \leq c(r)k D = c(r)k  n^{k-1-\alpha}(\log\log{n})^{\frac{1}{4(k-1)}} \leq n^{k-1-\alpha+ \frac{1}{2k^2}}.$$ 
Also since $n$ is large, $\frac{|H|}{kD}\geq {{n-1}\choose{k-1}} ( k n^{k-1-\alpha}(\log\log{n})^{\frac{1}{4(k-1)}})^{-1} \geq n^{\alpha/2}$ holds. Thus for large enough $n$,

\begin{align*} 
\frac{6n^{k/(k-1)}D^{(k-2)/(k-1)}}{(\log\log{\frac{|H|}{k D}})^{\frac{1}{2(k-1)}}} \leq \frac{6n^{k/(k-1)}D^{(k-2)/(k-1)}}{(\log\log{ n^{\alpha/2}})^{\frac{1}{2(k-1)}}} 
\stackrel{(\ref{def D})}{\leq} n^{k-1 + \frac{1- (k-2)\alpha}{k-1} }(\log\log{n})^{-\frac{1}{4(k-1)}}.
\end{align*}

Last, every edge in $H\setminus (H_0\cup H_1)$ contains two vertices of $T$ and $k-2$ vertices of $V(H)$. By (\ref{t size}),
 $$|H\setminus (H_0\cup H_1)| \leq {{|T|}\choose 2}n^{k-2} \stackrel{\eqref{t size}}{\leq} n^{k-2+2\alpha}(\log\log{n})^{-\frac{1}{2(k-1)}}.$$ \end{proof}

\begin{claim} \label{H_1}
$$|H_1| \leq |G| +  n^{k-1- 2^{-\ell+2}+2\alpha}(\log\log{n})^{-\frac{1}{2(k-1)}}.$$
\end{claim}
\begin{proof}
Note that $k\geq 2^{\ell-1}+2\geq l+2$. We consider $H_1\setminus G$. For an edge $e$ in $H_1\setminus G$ it satisfies $|e\cap (V(H)\setminus T)|=k-1$ and a $(k-1)$-set $e\setminus T$ lies in at least two edges of $H$.
For each pair $\{u,u'\} \subseteq T$, we consider the $(k-1)$-uniform hypergraph 
$$H^{\{u,u'\}} = \{ e' : e'\cup \{u\} \in E(H_1\setminus G), e'\cup \{u'\} \in E(H_1\setminus G), e'\subseteq V(H)\setminus T \}.$$ By the definition of $G$, every edge in $H_1\setminus G$ belongs to $H^{\{u,u'\}}$ for at least one pair $\{u,u'\}$. However, if $H^{\{u,u'\}}$ contains a copy of $H'(k-1,\ell-2)$, then the copy together with $u,u'$ form a copy of $H'(k,\ell-1)$ in $H$, which gives us an $r$-regular subgraph of $H$. Thus Proposition \ref{ext(H(k,l))} and the fact that $k-1= 2^{\ell-1}+1\geq \ell-2+3$ for $\ell\geq 2$ imply $|H^{\{u,u'\}}|\leq 2 n^{k-1-2^{-\ell+2}}$.
Thus we get $$|H_1\setminus G|\leq \sum_{ \{u,u'\} \in {{T}\choose{2}}} |H^{\{u,u'\}}| \leq {{|T|}\choose 2} 2 n^{k-1-2^{-\ell+2}} \stackrel{(\ref{t size})}{\leq} n^{k-1-2^{-\ell+2}+2\alpha }(\log\log{n})^{-\frac{1}{2(k-1)}}.$$ \end{proof}

\noindent If $k> 2^{\ell-1}+2$, we choose $\alpha=\frac{1}{2(k-2)}+ 2^{-\ell}$, then we get $$\max\{k-1-\alpha + \frac{1}{2k^2},~ k-1+ \frac{1-(k-2)\alpha}{k-1},~ k-1-2^{-\ell+2}+2\alpha,~ k-2+2\alpha \} \leq k-1 - \frac{1}{ 2r^2-2}.$$
Hence, $|H\setminus G| = |H_0|+|H_1\setminus G| + |H\setminus (H_0\cup H_1)| \leq 3n^{k-1-\frac{1}{2r^2-2}}.$
We conclude that for large enough $n$, $$|H| \leq {{n-1}\choose {k-1}} + 3 n^{k-1-\frac{1}{2r^2-2}}.$$
\noindent If $k= 2^{\ell-1}+2$, then we choose $\alpha=2^{-\ell+1} = \frac{1}{k-2}$, then we get 
 $|H\setminus G| = |H_0|+|H_1\setminus G| + |H\setminus (H_0\cup H_1)| \leq 3 n^{k-1}(\log\log{n})^{-\frac{1}{4(k-1)}}$ and we conclude that for large enough $n$, $$|H| \leq {{n-1}\choose {k-1}} + 3 n^{k-1}(\log\log{n})^{-\frac{1}{4(k-1)}}$$
This contradicts our initial assumption. Therefore, the Theorem holds.
 \end{proof}

\begin{remark}\label{alpha}
If $k \geq 2^{\ell}+3$, then we may choose $\alpha:= \frac{3\cdot 2^{\ell}+4}{2^{\ell}(3\cdot 2^{\ell}+5)}, D:= n^{k-1-\alpha}$ and go through the argument above. Then we can conclude $|H\setminus G| \leq 3 n^{k-1-\frac{1}{2^{\ell-1}(3\cdot 2^{\ell}+5)}}$ for any $n$-vertex $k$-uniform hypergraph $H$ with no $r$-regular subgraphs when $n$ is large enough. In order to get Theorem \ref{stability}, we assume $$\alpha:= \frac{3\cdot 2^{\ell}+4}{2^{\ell}(3\cdot 2^{\ell}+5)},\enspace D:= n^{k-1-\alpha},\enspace \ell:=\lceil\log{r} \rceil$$ throughout the paper.
\end{remark}

\section{ Asymptotic structure of $H$}\label{section stability}
In this section, we want to show that the asymptotic structure of $H$ is close to a full $k$-star. We let $G$ be as we define in the previous section, and $\alpha= \frac{3\cdot 2^{\ell}+4}{2^{\ell}(3\cdot 2^{\ell}+5)}, D= n^{k-1-\alpha}, \ell=\lceil\log{r} \rceil$ as in Remark \ref{alpha} and $T$ denote the set of vertices of $H$ of degree at least $D$. Then we still have (\ref{t size}). We also define
$$ G':=\{e\setminus T : e\in G\}.$$
By definition of $G$, we have 
\begin{align}\label{G,G'}
|G|=|G'|, \enspace E(G') = \bigcup_{x\in T} E(G_x).
\end{align}
where $G_x$ is the link graph of $x$ in $G$ (i.e. $G_x =\{e\setminus \{x\} : x\in e, e\in G \}$).
In order to prove Theorem \ref{stability}, we count the copies of $H(k-1,\ell+1)$ in $G'$ and show that there exists a vertex $v$ such that almost all copies of $H(k-1,\ell+1)$ consist of $(k-1)$-sets in $G_v$. We define $\beta := k^{4k} n^{-\frac{1}{2^{\ell-1}(3\cdot 2^{\ell}+5)}}$ and use it throughout the paper.

\begin{theorem} \label{stability}
For integers $k,r,\ell$ with $\ell=\lceil \log{r} \rceil, k\geq 2^\ell+3$, there exists an integer $n_k$ such that the following holds.
If $n\geq n_k$ and $H$ is an $n$-vertex $k$-uniform hypergraph $H$ with no $r$-regular subgraphs such that $|H| \geq {{n-1}\choose {k-1}} - \beta n^{k-1}/k^{4k}$, then there exists a vertex $v$ in $H$ such that $$|G_{v}| \geq (1- \beta)|G'|$$ with $\beta =  k^{4k} n^{-\frac{1}{2^{\ell-1}(3\cdot 2^{\ell}+5)}}.$
\end{theorem}
\begin{proof}
We take a $k$-uniform hypergraph $H$ with no $r$-regular subgraphs satisfying $|H| \geq {{n-1}\choose {k-1}} - n^{k-1-\frac{1}{2^{\ell-1}(3\cdot 2^{\ell}+5)}}.$
Then by Remark \ref{alpha}, we know 
\begin{equation}\label{assump}
|G|=|G'| \geq  {{n-1}\choose {k-1}} - 4 n^{k-1-\frac{1}{2^{\ell-1}(3\cdot 2^{\ell}+5)}}.
\end{equation}
By Theorem~\ref{asymptotic claim} and the fact $D=n^{k-1-\alpha}$, we also have
\begin{align}\label{t size 2}
t \leq D^{-1}k|H| \leq n^{\alpha}.
\end{align}
We pick $v$ such that $$ |G_{v}| = \max_{x\in V(H)} |G_{x}|.$$ For a contradiction, we assume $|G_{v}| < (1- \beta)|G'|$. For each $(k-1)$-set $e$ in $G'$, we define $g(e):= x$ if $e \in G_x$.
Let 
$$R_i(G'):=\{ \{f_1,f_2\} : f_1,f_2\in E(G'), |f_1\cap f_2|=i\}, \enspace R'_i(G'):=\{ \{f_1,f_2\} \in R_i(G') : g(f_1)\neq g(f_2)\}.$$
Also we let 
$$R'(G') := R'_{\ell}(G')\cup R'_{\ell+1}(G').$$ 
For a hypergraph $F$, we define $P(F)$ to be the set of copies of $H(k-1,\ell+1)$ in $F$ as follows.
$$P(F) := \{H' : H'\subseteq F, H'\simeq H(k-1,\ell+1)\} .$$ 
Also let 
$$P_1(G'):=\{H' \in P(G'): \exists\{f_1,f_2\}\in R'(G') \text{ such that } \{f_1,f_2\}\subseteq H', f_1\cap f_2 \subseteq V_d(H')\},$$
$$P_0(G'):= P(G')\setminus P_1(G').$$
%$P_1(G')$ be the set of copies of $H(k-1,\ell+1)$ so that the copy contains a pair $\{f_1,f_2\}$ in $R'(G')$ in the way that all vertices in $f_1\cap f_2$ are dynamic vertices in the copy of $H(k-1,\ell+1)$. 
%Let $P_1(G')$ be the set of copies of $H(k-1,\ell+1)$ so that the copy contains at least one pair $\{f_1,f_2\}$ in $R'(G')$ in the way that all vertices in $f_1\cap f_2$ are dynamic vertices of the copy of $H(k-1,\ell+1)$. 
Let $K$ be the complete $(k-1)$-graph on $V(G')$. To count the number of copies of $H(k-1,\ell+1)$ in $K$, we choose two disjoint $(k-1)$-sets, and choose $\ell+1$ vertices from one part and match them with other $\ell+1$ vertices on the other part. In this manner, one copy of $H(k-1,\ell+1)$ is counted exactly $2^{\ell+1}$ times which is the number of pairs in $H(k-1,\ell+1)$. Thus we get
\begin{equation} \label{PK}
|P(K)| = \frac{1}{2^{\ell+2}} {{k-1}\choose {\ell+1}} \frac{(k-1)!}{(k-\ell-2)!} {{n-1}\choose {k-1}}{{n-k}\choose {k-1}} .
\end{equation} 
Since $n$ is large enough and $(\frac{n}{k})^k \leq {{n}\choose{k}}$ holds,
\begin{align}\label{PK lower bound}
|P(K)| \geq \frac{1}{2^{\ell+2}} {{n-1}\choose {k-1}}{{n-k-1}\choose {k-1}} \geq  k^{-2k} n^{2k-2}.
\end{align}
Also (\ref{assump}) implies  
 \begin{equation}\label{K-G'} |K\setminus G'| \leq 4 n^{k-1-\frac{1}{2^{\ell-1}(3\cdot 2^{\ell}+5)}}.\end{equation}
First we show a lower bound on $|P(G')|$.

\begin{claim} \label{P(G')}
$|P(G')| > (1-\beta)|P(K)|$.
\end{claim}
\begin{proof}
It is enough to show $|P(K)\setminus P(G')| < \beta |P(K)|$. 
Note that any copy of $H(k-1,\ell+1)$ in $P(K)\setminus P(G')$ contains an edge $e \in K\setminus G'.$ Also by (\ref{H(k,l) disjoint}), we can find another edge $e'$ satisfying 
\begin{equation*}
\begin{minipage}[c]{0.7\textwidth}
$e'\in H(k-1,\ell+1)$ with $e\cap e'=\emptyset$ and $e\cup e' = V(H(k-1,\ell+1))$.
\end{minipage}
\end{equation*}
Thus, in order to count $P(K)\setminus P(G')$, we take a $(k-1)$-set $e$ in $K\setminus G'$ and a $(k-1)$-set $e'$ in $K$ disjoint from $e$. 
There are $|K\setminus G'|$ ways to choose $e$, and for fixed $e$, there are ${{n-k}\choose{k-1}}$ ways to choose $e'$. Then there are at most ${{k-1}\choose {\ell+1}} \frac{(k-1)!}{(k-\ell-2)!}$ copies of $H(k-1,\ell+1)$ containing both $e,e'$ (due to different ways to pair up dynamic vertices). By (\ref{K-G'}), the definition of $\beta$, and the fact that $(\frac{n}{k-1})^{k-1}\leq {{n}\choose{k-1}}$,
we have $$|K\setminus G'|\leq 4 n^{k-1 - \frac{1}{2^{\ell-1}(3\cdot 2^{\ell}+5)}} = 4 k^{-4k} \beta n^{k-1} < \frac{\beta}{2^{\ell+2}} {{n-1}\choose{k-1}},$$
Thus,
\begin{align*}
|P(K)\setminus P(G')|&\leq {{k-1}\choose {\ell+1}} \frac{(k-1)!}{(k-\ell-2)!}|K\setminus G'| {{n-k}\choose {k-1}} \\
& < \frac{\beta }{2^{\ell+2}} {{k-1}\choose {\ell+1}} \frac{(k-1)!}{(k-\ell-2)!} {{n-1}\choose {k-1}}{{n-k}\choose {k-1}}
\stackrel{(\ref{PK})}{=} \beta |P(K)|.
\end{align*}
  \end{proof}
Now we estimate $|P(G')|$ to show a contradiction.
\begin{claim} \label{P_1}
$|P_1(G')| < \frac{1}{2}\beta |P(K)|$.
\end{claim}
\begin{proof}
First, we count the number of pairs $\{f_1,f_2\}$ in $R'(G')=R'_\ell(G')\cup R'_{\ell+1}(G')$.
For $i\in \{\ell,\ell+1\}$, we take two disjoint $(k-1-i)$-sets $e_1,e_2$, and two distinct vertices $x,y \in T$. Let 
$$p(\{e_1,e_2\},\{x,y\}):=\{ h: \{e_1\cup h, e_2\cup h\}\in R'_i(G'), g(e_1\cup h)=x, g(e_2\cup h)=y\}.$$
If $p(\{e_1,e_2\},\{x,y\})$ contains a copy of $H(i,\ell-1)$, \eqref{H(k,l) regular subgraph} gives us $2^{\ell-1}$ edge-disjoint matchings of size two in $p(\{e_1,e_2\},\{x,y\})$ covering the same ground set. For each of these matchings $\{m,m'\}$, we obtain two edge-disjoint matchings $\{m\cup e_1\cup \{x\}, m'\cup e_2\cup \{y\}\}$ and $\{m\cup e_2\cup \{y\} , m'\cup e_1\cup \{x\}\}$. Thus we get $2^{\ell} \geq r$ edge-disjoint matchings covering the same ground set. It is a contradiction since $H$ does not contain any $r$-regular subgraphs. Thus $p(\{e_1,e_2\},\{x,y\})$ does not contain $H(i,\ell-1)$, so it has at most $2n^{i-\frac{1}{2^{\ell-1}}}$ edges by Proposition \ref{ext(H(k,l))}.
There are at most ${{n-1}\choose {k-i-1}}^2$ choices for $\{e_1,e_2\}$ and ${ {|T|}\choose 2}$ choices for $\{x,y\}$. So, 

\begin{align}\label{R'iG'}
|R'_i(G')| \leq \sum_{\{e_1,e_2\},\{x,y\}}2n^{i-\frac{1}{2^{\ell-1}}} \leq  2n^{i-\frac{1}{2^{\ell-1}}}{{n-1}\choose {k-i-1}}^2 {{|T|} \choose 2}.
\end{align}

For each pair in $R'_i(G')$, we can complete a copy of $H(k-1,\ell+1)$ by adding $i$ more vertices from outside to play the role of dynamic verticese and choosing $\ell+1-i$ vertices from each of $e_1,e_2$ to play the role of dynamic vertices and match those dynamic vertices. 
Thus each pair in $R'_i(G')$ is contained in at most $(\ell+1)!{{k-i-1}\choose {\ell+1-i}}^2{{n-1}\choose {i}}$ copies of $H(k-1,\ell+1)$. Thus by the fact that $2\alpha -\frac{1}{2^{\ell-1}} = -\frac{1}{2^{\ell-1}(3\cdot 2^{\ell}+5)}$ and the definition of $\beta$,
{
\begin{eqnarray*}
|P_1(G')| &=& \sum_{i=\ell}^{\ell+1}|R'_{i}(G')|(\ell+1)!{{k-i-1}\choose {\ell+1-i}}^2{{n-1}\choose {i}}  \\
&\stackrel{\eqref{R'iG'}}{\leq}& \sum_{i=\ell}^{\ell+1} 2n^{i-\frac{1}{2^{\ell-1}}}{{n-1}\choose {k-i-1}}^2 {{|T|} \choose 2} (\ell+1)! {{k-i-1}\choose{\ell+1-i}}^2 {{n-1}\choose {i}} \\
&\stackrel{(\ref{t size})}{\leq}& \sum_{i=\ell}^{\ell+1} k^2(\ell+1)! n^{i- \frac{1}{2^{\ell-1}} + 2\alpha}{{n-1}\choose {k-i-1}}^2{{n-1}\choose {i}} \\
&\leq& \sum_{i=\ell}^{\ell+1} k^2(\ell+1)!  n^{i- \frac{1}{2^{\ell-1}} + 2\alpha} n^{2k-2i-2} n^{i} \\
  &\leq& 2 k^2 k! n^{2k-2 - \frac{1}{2^{\ell-1}} + 2\alpha } < \frac{1}{2} k^{4k} n^{-\frac{1}{2^{\ell-1}(3\cdot 2^{\ell}+5)}}  k^{-2k} n^{2k-2} \stackrel{(\ref{PK lower bound})}{\leq} \frac{1}{2}\beta |P(K)|
\end{eqnarray*}
}
\end{proof}

Before we estimate $|P_0(G')|$, we prove the following claim.
\begin{claim}\label{g-Wedge}
Let $H'$ be a copy of $H(k-1,\ell+1)$ in $P_0(G')$. Then there exists a vertex $x\in T$ so that every $(k-1)$-set $e$ in $H'$ is contained in $G_x$.
\end{claim}
\begin{proof}
Remind that $V_d(H')$ denotes the set of dynamic vertices of $H'$.
We consider a graph $G_{H'}$ such that 
\begin{align*}
V(G_{H'})&:= \{f\in H'\}, \\ 
E(G_{H'})&:= \{ f_1f_2 : f_1,f_2 \in H', |f_1\cap f_2|\in \{\ell,\ell+1\}, f_1\cap f_2 \subseteq V_d(H')\}.
\end{align*} 
Let $A,B$ be two stationary parts in $H'$.
Consider two $(k-1)$-sets $f_1,f_2$ in $H'$ such that $|f_1\setminus f_2|=|f_2\setminus f_1|=1$ and both $f_1,f_2$ contain $A$. We consider an edge $e\in H'$ such that $e=(f_1\setminus A)\cup B $. Then, $e$ does not share any stationary vertices with $f_1$ or $f_2$, $|f_1\cap e|=\ell+1$, and $|f_2\cap e|=\ell$. Thus $e$ is adjacent to both $f_1$ and $f_2$ in $G_{H'}$.
Thus any two $(k-1)$-sets $f,f'$ in $H'$ containing $A$ with $|f\setminus f'|=1$ are in the same component of $G_{H'}$. Since being in the same component is transitive, all $(k-1)$-sets in $H'$ containing $A$ are in the same component in $G_{H'}$. By the same logic, all $(k-1)$-sets containing $B$ are in the same component in $G_{H'}$. Also there are edges between $f_1$ and $e$, so $G_{H'}$ is connected. On the other hand, if two $(k-1)$-sets $f_1,f_2$ are adjacent in $G_{H'}$, then $g(f_1)=g(f_2)$ because of the definition of $P_0(G')$. This fact and connectedness of $G_{H'}$ together imply that there exists a vertex $x\in T$ such that every edge in $H'$ belongs to $G_x$.\end{proof}

\begin{claim}\label{P_0}
$|P_0(G')| < (1- \frac{3}{2}\beta) |P(K)|$.
\end{claim}
\begin{proof}
By Claim~\ref{g-Wedge} and \eqref{H(k,l) regular subgraph}, a copy of $H(k-1,\ell+1)$ in $P_0(G')$ consists of $2^{\ell+1}$ pairs of two disjoint $(k-1)$-sets all in $G_x$ for some $x\in T$. To count the number of copies of $H(k-1,\ell+1)$ in $P_0(G')$, we choose two disjoint $(k-1)$-sets $e,e'$ with $g(e)=g(e')$, and choose $\ell+1$ elements from one and match them with $\ell+1$ vertices in the other side to play the role of dynamic vertices. Also, each $H(k-1,\ell+1)$ is counted $2^{\ell +1}$ times from each pair in this counting.
So, 
\begin{equation}\label{eq P_0} |P_0(G')| \leq \frac{1}{2^{\ell+1}}{{k-1}\choose {\ell+1}} \frac{(k-1)!}{(k-\ell-2)!} \sum_{x\in T'} {{|G_x|}\choose 2}.
\end{equation}
By convexity, the right side of (\ref{eq P_0}) is maximized when $|G_{v}| = (1-\beta) |G'|$ and $|G_{u}| = \beta|G'|$ for another vertex $u \in T$ and $|G_{x}|=0$ for other $x$. And ${{|G'|}\choose {2}} \leq \frac{1}{2}{{n-1}\choose {k-1}}^2 \leq  \frac{1}{2}(1+ \frac{k}{n}){{n-1}\choose {k-1}}{{n-k-1}\choose {k-1}}$.
Because of the fact $\frac{k}{n} + 2\beta^2 <\frac{1}{2}\beta$,
$$ |P_0(G')| \stackrel{\eqref{PK}}{\leq} ((1-\beta)^2 + \beta^2)(1+ \frac{k}{n})|P(K)| \leq (1-\frac{3}{2}\beta)|P(K)|.$$\end{proof}
\noindent In total, $$|P(G')| = |P_0(G')|+|P_1(G')| < 
( 1-\frac{3}{2}\beta)|P(K)| + \frac{1}{2}\beta|P(K)|\leq (1-\beta) |P(K)|.$$ 
However, it contradicts to Claim \ref{P(G')}.
Therefore, we conclude $|G_{v}| \geq (1-\beta)|G'|$. This proves Theorem \ref{stability}.
\end{proof}

Note that $2^{\ell}+3\geq 2r+1$ and $n^{-\frac{2}{3r^2}} \leq k^{4k} n^{-\frac{1}{2^{\ell-1}(3\cdot 2^{\ell}+5)}} \leq n^{-\frac{1}{6r^2}}$ for large enough $n$. Thus Theorem \ref{asymptotic}, Remark \ref{alpha} and Theorem \ref{stability} together imply Theorem \ref{asymptotic theorem}. If $r\in \{3,4\}$, then $\ell=2$, so we get $-\frac{1}{2^{\ell-1}(3\cdot 2^{\ell}+5)}= -\frac{1}{34},$ thus $\beta= k^{4k}n^{-\frac{1}{34}}$. So the above proof actually gives the following.

\begin{remark}\label{beta}
If $r\in \{3,4\}$ and $k\geq 2r+1$, then there exists an integer $n_k$ such that for $n>n_k$ any $n$-vertex $k$-uniform hypergraph $H$ with no $r$-regular subgraphs $H$ with $|H|\geq {{n-1}\choose{k-1}} - n^{k-1-\frac{1}{34}}$, then there exists a vertex $v$ which belongs to at least $(1- k^{4k} n^{-\frac{1}{34}})|H|$ edges.
\end{remark}

\section{Proof of Theorem \ref{main theorem}} \label{exact result}

In this section, we prove Theorem \ref{main theorem}. We assume $r\in \{3,4\}$, $k =k'r $, $k'\geq 140$ and that $n$ is sufficiently large.
If an $n$-vertex $k$-uniform hypergraph $H$ with no $r$-regular subgraphs contains at least ${{n-1}\choose {k-1}}$ edges, we may suppose $|H|={{n-1}\choose{k-1}}$ by deleting some edges if necessary. If we show that $H$ has to be a full $k$-star, then it completes the theorem because a full $k$-star with one more edge $e$ always contains an $r$-regular subgraph when $r\mid k$ by the following Observation \ref{obs}. 
Since $r\in \{3,4\}$ implies $\ell=2$, we have $\beta = k^{4k} n^{-\frac{1}{34}}$.
By Remark \ref{beta}, there exists a vertex $v$ with $|H^*|\leq k^{4k} n^{-\frac{1}{34}} {{n-1}\choose {k-1}}$, where
we define
$$H^* := \{e\in E(H): v\notin e\},\enspace \tilde{H}:= \{ f\in {{V(H)}\choose {k}} : v\in f, f\notin H\}.$$ 
Then 
\begin{align}\label{H* size}
|H^*|=|\tilde{H}| \leq k^{4k} n^{-\frac{1}{34}} {{n-1}\choose {k-1}}
\end{align}
by our assumption and Remark \ref{beta}. To show that $H$ is a full $k$-star, it is enough to show $|H^*|=0$. Suppose $|H^*|>0$ for a contradiction.

\begin{observation}\label{obs}
If $\{e'_1,e'_2,\cdots,e'_r\}$ is a partition of $e\in H^*$ into $r$ sets of size $k'$ and $g \subseteq V(G)-\{v\}-e$ is a $(k'-1)$-set, then there exists $j$ such that $(e\setminus e'_j)\cup g\cup \{v\}$ is not an edge of $H$.
\end{observation}
\begin{proof}
Suppose not. Then $e, (e\setminus e'_1)\cup g\cup \{v\}, (e\setminus e'_2)\cup g\cup \{v\}, \cdots, (e\setminus e_r)\cup g\cup \{v\}$ together form an $r$-regular subgraph, a contradiction. Thus there is a choice $j$ such that $(e\setminus e'_j)\cup g\cup \{v\}$ is not an edge of $H$. \end{proof}

Here we define {\em wedge} as follows and count them to derive a contradiction.

\begin{definition}\label{def wedge}
A pair of $k$-sets $(e,f)$ is a wedge if it satisfies the following:

(1) $e\in H^*, f\in \tilde{H}$;

(2) $|e\cap f| = k-k'$.
\end{definition}

Let $\Lambda(H)$ be the number of wedges in $H$.
Then the following claim gives us a lower bound for $\Lambda(H)$.

\begin{claim}\label{lambda lower bound}
$$\Lambda(H) \geq \frac{1}{r}{{k}\choose{k'}}{{n-k-1}\choose{k'-1}}|H^*|.$$
\end{claim}
\begin{proof}
To count the wedges $(e,f)$ in $H^*$, instead we count $(e, f\setminus(e\cup \{v\}), e\setminus f)$.
First we choose $e$. There are $|H^*|$ ways to choose $e$. For each chosen $e$, there are ${{n-k-1}\choose{k'-1}}$ ways to choose $(k'-1)$-set $S$ outside $e\cup\{v\}$ playing the role of $f\setminus (e\cup \{v\})$. 
For fixed $e$ and $S$, we call a $k'$-subset $D$ of $e$ {\em $(S,e)$-good} if $(e\setminus D)\cup S\cup \{v\} \in H$, and {\em $(S,e)$-bad} otherwise.
By Observation \ref{obs}, for an $r$-equipartition $\{e'_1,\dots, e'_r\}$ of $e$, at least one of $e'_i$ is $S$-bad. By Observation \ref{obs2}, this implies that for fixed $e$ and $S$, there are at least $\frac{1}{r} {{k}\choose{k'}}$ $(S,e)$-bad subsets of $e$.
For each $(S,e)$-bad subset $D$ of $e$, we get a wedge $$(e, (e\setminus D )\cup S\cup \{v\}).$$
Since those wedges are distinct for distinct $(e,S,D)$, we get
$\Lambda(H) \geq \frac{1}{r}{{k}\choose{k'}}{{n-k-1}\choose{k'-1}}|H^*|.$
\end{proof}
\begin{definition}
A $3$-set $T\in {{V(H)\setminus\{v\}}\choose {3}}$ is good if $d_{\tilde{H}}(T\cup \{v\}) < \frac{1}{8}{{n-k-4}\choose{k-4}}$ and bad otherwise. Let $W$ be the collection of all bad $3$-sets in $H$.
\end{definition}

\begin{claim}
$$|W| \leq k^{5k} n^{3 -\frac{1}{34}}.$$
\end{claim}
\begin{proof}
We count all $k$-sets in $\tilde{H}$ which contain a bad $3$-set. 
Each bad $3$-set $T$ belongs to at least $\frac{1}{8} {{n-k-4}\choose{k-4}}$ distinct $k$-sets in $\tilde{H}$. Also, each $k$-set can contain at most ${{k}\choose{3}}$ distinct bad $3$-sets.
Thus the number of $k$-sets in $\tilde{H}$ containing a bad $3$-set is at least $$\frac{1}{8}{{k}\choose{3}}^{-1}|W|{{n-k-4}\choose{k-4}}.$$ 
From (\ref{H* size}),
$$\frac{1}{8}{{k}\choose{3}}^{-1}|W|{{n-k-4}\choose {k-4}} \leq |\tilde{H}| \leq k^{4k} n^{-\frac{1}{34}}{{n-1}\choose{k-1}}.$$
Since ${{n-1}\choose{k-1}} \leq 2 k^3 n^3{{n-k-4}\choose{k-4}}$ for sufficiently large $n$, we get
$$|W| \leq 8{{k}\choose{3}} k^{4k} n^{-\frac{1}{34}} {{n-1}\choose{k-1}}{{n-k-4}\choose{k-4}}^{-1}\leq \frac{16}{6}k^{4k+6}n^{3-\frac{1}{34}}  \leq k^{5k} n^{3 - \frac{1}{34}}$$ since $n$ is sufficiently large and $k\geq 140r$.
\end{proof}

\begin{claim}\label{lambda upper bound}
$$\Lambda(H) \leq 1.01|\tilde{H}|{{k-1}\choose{k'-1}} \binom{k'-34}{3}^{-1} {{n}\choose{k'-3}}n^2.$$
\end{claim}
\begin{proof}
To count the number of wedges $(e,f)$ in $H$, instead we count $(f, e\cap f, e\setminus f)$. 
The number of ways to pick $f$ is $|\tilde{H}|$. For fixed $f$, the number of ways to choose a $(k-k')$-subset $D$ of $f\setminus \{v\}$ which will play a role of  = $e\cap f$ is 
$${{k-1}\choose{k-k'}} = {{k-1}\choose{k'-1}}.$$
For a $(k-k')$-set $D$, let $H^*_D$ be the link graph of $D$ in $H^*$(i.e.  $H^*_D= \{e\setminus D : e \in H^*, D\subseteq e \}$).
We partition $H^*_D$ into the following two hypergraphs,
 \begin{align*}
 H_D^1 &:= \{ B \in H^*_D: B \text{ contains at least 35 disjoint bad $3$-sets} \} \\
 H_D^2 &:= H^*_D\setminus H_D^1.
 \end{align*}

First, since any $k'$-set in $H_D^1$ contains $35$ disjoint bad $3$-sets, $k'\geq 140$ and $n$ is sufficiently large, 
\begin{align}\label{HD1 size}
|H_D^1| \leq |W|^{35}{{n}\choose{k'-105}} \leq (k^{5k} n^{3-\frac{1}{34}})^{35} n^{k'-105}\leq k^{165k} n^{k'-1-\frac{1}{34}} <  \frac{1}{100 k^k} n^{k'- 1}.
\end{align}
To find an upper bound of $|H_D^2|$, note that any $k'$-set in $H_D^2$ contains at least one good $3$-set because it contains at most $34$ disjoint bad $3$-sets and $k'-3\cdot 34 \geq 3$. Thus we first bound the number of pairs $(A,T)$ where $T$ is a good $3$-set and $A=B\setminus T$ for some $B\in H_D^2$. There are at most ${{n}\choose{k'-3}}$ ways to choose $A$. 
We claim that for fixed $A$, there are at most $2{{n-k+k'-1}\choose{2}}$ distinct good $3$-sets $T$ such that $A\cup T \in H_D^2$. Otherwise, by Theorem \ref{generalized C4}, there exists four good $3$-sets $T_1,T_2,T_3,T_4$ with 
$$A\cup T_i \in H_D^2,~ T_1\cup T_2 = T_3\cup T_4=T',~ T'\cap (D\cup\{v\})=\emptyset ~\text{ and }~ T_1\cap T_2=T_3\cap T_4=\emptyset.$$ 
Since each $T_i\cup \{v\}$ belongs to at most $\frac{1}{8}{{n-k-4}\choose{k-4}}$ sets in $\tilde{H}$, there are at most $\frac{1}{2}{{n-k-4}\choose{k-4}}$ many $(k-4)$-sets $S$ outside $D\cup A \cup T' \cup \{v\}$ such that $T_i\cup S\cup \{v\} \notin H$ for some $i$. So there exists a $(k-4)$-set $S$ such that $T_i\cup S\cup \{v\} \in H$ for $i\in [4]$ and $S\cap (D\cup A\cup T'\cup \{v\}) =\emptyset$. Then 
$$ T_1\cup S\cup \{v\}, \dots , T_4\cup S\cup \{v\},~ D\cup A\cup T_1, \dots, D\cup A\cup T_4$$
together contain both $3$-regular subgraph and $4$-regular subgraph of $H$, a contradiction. Thus for each $A$, there are at most $2{{n-k+k'-1}\choose{2}} \leq n^2$ distinct $T$'s with $A\cup T \in H_D^2$. Thus the number of such pairs $(A,T)$ is at most $\binom{n}{k'-3} n^2 $. 

Let $B \in H_D^2$, then $B$ does not contain a matching of bad sets of size $35$, and $k'\geq 4\cdot 35 = 140$. So we apply Theorem \ref{obs 3matching}, then we get that the number of bad sets in $B$ is at most $\binom{k'}{3} - \binom{k'-34}{3}$. Thus there are at least $\binom{k'-34}{3}$ good sets in $B$. Hence each $B$ yields at least $\binom{k'-34}{3}$ distinct pairs $(A,T)$. So
\begin{align}\label{HD2 size}
|H_D^2| \leq \binom{k'-34}{3}^{-1}\binom{n}{k'-3}n^2.
\end{align}
Since $(\frac{n}{k'-1})^{k'-1} \leq \binom{n}{k'-3}$ and $k\geq 3k'$, we know $ \frac{1}{ k^k} n^{k'- 1} \leq \binom{k'-34}{3}^{-1} \binom{n}{k'-3}n^2 $ for large enough $n$. Thus from (\ref{HD1 size}) and (\ref{HD2 size}),
\begin{align*}
|H^*_D| &= |H_D^1|+|H_D^2| \leq \frac{1}{100 k^k} n^{k'- 1} + \binom{k'-34}{3}^{-1} \binom{n}{k'-3}n^2 \leq 1.01\binom{k'-34}{3}^{-1} \binom{n}{k'-3}n^2
\end{align*}
for sufficiently large $n$.
Therefore we get
$$ \Lambda(H) \leq 1.01|\tilde{H}|{{k-1}\choose{k'-1}}\binom{k'-34}{3}^{-1} \binom{n}{k'-3}n^2.$$
\end{proof}

Therefore, from Claim \ref{lambda lower bound} and Claim \ref{lambda upper bound}, we get 
\begin{align*}
\frac{1}{r}{{k}\choose{k'}}{{n-k-1}\choose{k'-1}}|H^*| \leq \Lambda(H) \leq 1.01|\tilde{H}|{{k-1}\choose{k'-1}}\binom{k'-34}{3}^{-1} \binom{n}{k'-3}n^2.
\end{align*}
Since $n$ is sufficiently large, we have ${{n}\choose{k'-3}}n^2 \leq 1.01 (k'-1)(k'-2){{n-k-1}\choose{k'-1}}$. Since we assumed $|H^*|=|\tilde{H}|>0$, this yields 
\begin{align*}
\frac{1}{r}{{k}\choose{k'}}{{n-k-1}\choose{k'-1}} &\leq 1.01{{k-1}\choose{k'-1}} {{k'-34}\choose{3}}^{-1} {{n}\choose{k'-3}}n^2\\
&\leq 1.01^2 {{k-1}\choose{k'-1}} {{k'-34}\choose{3}}^{-1} (k'-1)(k'-2) {{n-k-1}\choose{k'-1}}.
\end{align*}
and by dividing ${{n-k-1}\choose{k'-1}}$ on both sides, we get
\begin{align*}
\frac{1}{r}{{k}\choose{k'}} \leq  1.01^{2}{{k-1}\choose{k'-1}} {{k'-34}\choose{3}}^{-1} (k'-1)(k'-2) = \frac{k'}{k} {{k}\choose{k'}} \frac{1.01^2 \cdot 6 (k'-1)(k'-2)}{(k'-34)(k'-35)(k'-36)}.
\end{align*}
From this and the fact that $\frac{1}{r} = \frac{k'}{k}$, we get
\begin{align*}
\frac{(k'-34)(k'-35)(k'-36)}{(k'-1)(k'-2)} \leq 1.01^2 \cdot 6,
\end{align*}
which is a contradiction since $k'\geq 140$.

\section{What happens if $r$ is big or $r\nmid k$?} \label{example}
In the same spirit as Theorem \ref{main theorem}, we propose the following conjecture.

\begin{conjecture}\label{r>4}
For $r$, there exist $k_r$, $n_k$ such that 
for all $k> k_r$, and $n> n_k$, and $r\mid k$, if $H$ is a $k$-uniform hypergraph with no $r$-regular subgraphs, then $$|H|\leq {{n-1}\choose{k-1}}$$
and equality holds if and only if $H$ is a full $k$-star.
\end{conjecture}

The proof of Theorem \ref{main theorem} does not extend for the case $r>4$ because the author does not know how to generalize Theorem \ref{generalized C4} for more pairs of disjoint edges. However, if the following conjecture is true, then we can prove Conjecture \ref{r>4}.

\begin{conjecture}\label{r disjoint pairs}
For every positive integer $r$, there exist $k_r$, $n_k$ and $g(r)$ which satisfy the following.
 For $k\geq k_r, n\geq n_k$, any $n$-vertex $k$-uniform hypergraph $H$ with more than
$$ g(r){{n-1}\choose{k-1}}$$ edges contains distinct edges $A_1,B_1,\cdots, A_r, B_r$ so that $A_i\cap B_i =\emptyset$ for all $i=1,2\cdots, r$ and $A_1\cup B_1 =A_2\cup B_2 =\cdots =A_r\cup B_r.$
\end{conjecture}

Note that this conjecture is known to be true for $r=1,2$. For $r=1$, it's Erd\H{o}s-Ko-Rado Theorem. For $r=2$  F\"{u}redi \cite{Furedi} proved $k_2=3, g(2)\leq \frac{7}{2}$ and later Pikhurko and Verstra\"{e}te \cite{PV} improved it to $g(2)\leq \frac{7}{4}$.

In Theorem \ref{main theorem}, we assume that $k$ is much bigger than $r$ and $r\mid k$. What happens if the conditions do not hold?
First, let's see what happens if $k$ is not big enough in terms of $r$. The author believes that full $k$-star might be the only extremal example even when $k\geq 2r$ and $r\mid k$. However if $r=k$ then the extremal example is no longer only full $k$-star. Also, if $r>k$, then $|H|$ can be bigger than ${{n-1}\choose {k-1}}$. It is straightforward to check the following example.

\begin{example} \label{example A} Take an $n$-vertex full $k$-star $H$. We take a non-edge $e$ of $H$, and an edge $e'$ of $H$ such that $|e\cap e'|=k-1$. Then $(H \cup \{e\})\setminus\{e'\}$ does not have $r$-regular subgraphs if $r=k$, and $H\cup \{e\}$ does not have $r$-regular subgraphs if $r=k+1$.
\end{example}

As an example, if $r$ is bigger than $k$, even $r=k+1$ does not imply $|H|\leq {{n-1}\choose {k-1}}$ any more. However, as we can see in Section \ref{section asymptotic}, $|H|\leq (1+o(1)){{n-1}\choose{k-1}}$ still holds if $k \geq 2^{\lceil \log{r} \rceil-1} +2$. Thus, for $r = 2^l$ and $k\geq \frac{r}{2}+3 = 2^{l-1}+2$, the asymptotics of the number of edges in hypergraphs with no $r$-regular subgraphs is still $(1+o(1)){{n-1}\choose {k-1}}$ even though $r\geq k$. However, the following example shows that this becomes false if $r$ is much bigger.

\begin{example} \label{example B}
For an integer $c>1$, take an $n$-vertex $k$-uniform hypergraph $H$ such that $E(H)=\{e: e\in {{V(H)}\choose {k}}, |e\cap \{x_1,x_2,\cdots,x_c\}|=1 \}$. Then $|H|= c{{n-c}\choose {k-1}} \sim c{{n-1}\choose {k-1}}$. However, $H$ does not contain any $r$-regular subgraph for $r> c{{c(k-1)}\choose {k-2}} $.
\end{example}
\begin{proof}
Suppose $H$ contains an $r$-regular subgraph $R$, then $R$ must cover some vertices in $\{x_1,x_2,\cdots, x_c\}$. Assume it covers $\{x_1,x_2,\cdots, x_{c'}\}$. Since it must cover those vertices exactly $r$-times, $|R| = c'r$.
Then $V(R) = \frac{k|R|}{r} = c'k$. Then a vertex $x$ in $V(R)$ can be covered only by edges $e$ with $|e\cap (V(R)\setminus \{x_1,\cdots, x_{c}\}|=k-1$. So, degree of $x$ is at most $c' {{c'(k-1)}\choose {k-2}} <r$, a contradiction.\end{proof}

Hence, it is natural to ask the following question. 
Note that, such $r(k)$ must exist and $k\leq r(k) \leq 2{{2k-2}\choose {k-2}}+1$ by Theorem \ref{asymptotic} and Example \ref{example B}.

\begin{question} \label{question}
What is the minimum $r=r(k)$ such that 
$$\limsup_{n\rightarrow \infty} \frac{\max|H|}{{{n-1}\choose {k-1}}} >1$$
where the maximum is taken over all $n$-vertex $k$-uniform hypergraphs with no $r$-regular subraphs.
\end{question}

Now we consider the case where $r$ does not divide $k$ while $k$ is bigger than $r$. In \cite{MV2009}, Mubayi and Verstra\"{e}te conjectured the following.

\begin{conjecture}\cite{MV2009}\label{MV conjecture}
For every integer $k$ with $2\nmid k$, there exists an integer $n_k$ such that for $n\geq n_k$, if $H$ is an $n$-vertex $k$-uniform hypergraph with no $2$-regular subgraphs then $|H|\leq {{n-1}\choose {k-1}} + \lfloor \frac{n-1}{k} \rfloor$. Equality holds if and only if $H$ is a full $k$-star together with a maximal matching disjoint from the full $k$-star.
\end{conjecture}

In the same spirit, we may add more edges to full $k$-star when $r\geq 3$, $r<k$, $r\nmid k$. In order to construct an example, we need the following concept.

In 1973, Brown, Erd\H{o}s and S\'{o}s \cite{BES} proposed a study for a new parameter, $f_{k}(n,a,b)$, the largest number of edges in a $k$-uniform hypergraph on $n$ vertices that contains no $b$ edges spanned by $a$ vertices. Determining $f_k(n,a,b)$ for general tuple $(k,a,b)$ is very difficult. Note that finding value of $f_3(n,6,3)$ is known as the famous $(6,3)$-problem. In \cite{BES}, they showed the following.

\begin{theorem}\cite{BES}
If $a>k$ and $b>1$, then
$f_k(n,a,b) > c_{a,b} n^{\frac{kb-a}{b-1}}$.
\end{theorem}

Now we consider the following construction.

\begin{construction}\label{const}
Let $k=k'd,r=r'd$ be positive integers with $k\geq 3, r'\geq 3$ such that $k'$ and $r'$ are relatively prime. Consider a $(k-1)$-uniform $(n-2)$-vertex hypergraph $H'$ with $f_{k-1}(n-2,2k-2,r')$ edges such that $H'$ does not contain any $r'$ edges spanning at most $2k-2$ vertices. Especially, $|H'|$ contains at least $c_{k,r'} n^{\frac{r'-2}{r'-1}(k-1)}$ edges.

Now we consider two vertices $x,y$ disjoint from $V(H')$ and the hypergraph $H_{k,r}$ with 
$$ V(H_{k,r}) = V(H')\cup \{x,y\},$$
$$ E(H_{k,r}) = \{e\in {{V(H)}\choose k}: x\in e\} \cup \{ e\cup \{y\} : e\in E(H')\}.$$

Then $H_{k,r}$ contains at least ${{n-1}\choose {k-1}} + c_{k,r} n^{\frac{r'-2}{r'-1}(k-1)} $ edges, and $H_{k,r}$ contains no $r$-regular subgraphs.
\end{construction}
\begin{proof}
Assume that $H_{k,r}$ contains an $r$-regular subgraph $R$. Let $H^x$ be the full $k$-star in $H_{k,r}$ and $H^*$ be the hypergraph consisting edges not containing $x$. Since both $H^x$ and $H^*$ are subgraphs of two distinct full $k$-star, each of them does not contain any $r$-regular subgraph.
Thus $R$ must intersect both $H^x$ and $H^*$, thus $R$ must cover both $x$ and $y$. Since $R$ covers $x$ exactly $r$ times, $|R\cap H^x|=r$ and we have
 $$|R| = |R\cap H^x| +|R\cap H^*| = r+ |R\cap H^*| \geq r+1.$$ However, because $R$ induces an $r$-regular subgraph, $$r |V(R)|=k|R|= kr+k|R\cap H^*|.$$ Since $k',r'$ are relatively prime, $|R\cap H^*|$ must be a multiple of $r'$. Moreover $|R\cap H^*|\leq r$ because $y$ are not to be covered more than $r$-times. Hence $|V(R)| = \frac{k|R|}{r} \leq 2k$. Now we consider $\{ e-y : e\in R\cap H^*\}$. It is a set of at least $r'$ edges of $H'$ covering at most $2k-2$ vertices, a subset of $V(R)-\{x,y\}$. It is a contradiction to the definition of $H'$. Thus $H_{k,r}$ does not contain any $r$-regular subgraph. \end{proof}

Hence, there is an $n$-vertex $k$-uniform hypergraph $H$ with no $r$-regular subgraphs which contains quite more edges than ${{n-1}\choose {k-1}}$ if $r$ does not divide $k$. Hence we propose the following question.

\begin{question}
Determine the least value of $h(k,r)$ such that there exists a constant $c_{k,r}$ so that every $n$-vertex $k$-uniform hypergraph $H$ with no $r$-regular subgraphs satisfies 
 $$|H|\leq {{n-1}\choose {k-1}} + c_{k,r} n^{h(k,r)}.$$ 
\end{question}

The author suspects that $h(k,r)$ is related to the value of $gcd(k,r)$ based on the fact that the value we get from Construction \ref{const} is related to $k,r$, and $gcd(k,r)$.

Also, considering linear hypergraphs is another direction of studying regular subgraphs. The following question was proposed in \cite{DHLNPRSV}.

\begin{question}\cite{DHLNPRSV}\label{question DH}
For an integer $r$, let $f_{k,r}(n)$ be the maximum number of edges in a linear $n$-vertex $k$-uniform hypergraphs with no $r$-regular subgraphs. Is $f_{3,3}(n) = o(n^2)$?
\end{question}

Especially, authors of \cite{DHLNPRSV} asked if sufficiently large Steiner triple system contains a $3$-regular subgraph. In \cite{V}, Verstra\"{e}te observed that Lemma \ref{lemma 1} together with the fact that all linear $k$-uniform hypergraphs have maximum degree at most $\frac{n-1}{k-1}$ trivially imply the following.

\begin{corollary} \label{linear}
For any integers $k,r\geq 3$ and sufficiently large $n$, $$f_{k,r}(n) < 6n^2(\log\log n)^{-\frac{1}{2(k-1)}}.$$
\end{corollary}

Thus this answers Question \ref{question DH} and it implies that for an integer $r$, every $n$-vertex Steiner system contains an $r$-regular subgraph if $n$ is sufficiently large.

\section*{Acknowledgement}
The author is indebted to two anonymous referees, Alexandr V. Kostochka and Joonkyung Lee for very helpful comments and suggestions. The author especially thanks one referee for teaching him the current proof of Theorem \ref{main theorem} which is much better than the proof in the old version and yields better bound on $k$. The author is also grateful to Jacques Versatra{\"e}te for teaching him the implication from Lemma \ref{lemma 1} to Corollary \ref{linear}.

\medskip

{\footnotesize \obeylines \parindent=0pt

Jaehoon Kim
School of Mathematics
University of Birmingham
Edgbaston
Birmingham
B15 2TT
UK
}
\begin{flushleft}
{\it{E-mail addresses}:}
{\rm{j.kim.3@bham.ac.uk, mutualteon@gmail.com}}
\end{flushleft}

\end{document}